\documentclass[10pt]{article}
\usepackage{amsmath, amsthm, amssymb,todonotes, xcolor}
\usepackage{graphicx,float,authblk,mathtools,relsize}
\usepackage[skip=2pt,font=scriptsize]{caption}
\usepackage[all]{xy}
\xyoption{all}
\newtheorem{corollary}{Corollary}
\newtheorem{proposition}{Proposition}

\newtheorem{theorem}{Theorem}
\newtheorem{claim}{Claim}

\newtheorem*{theorem*}{Theorem}
\newtheorem*{thmA}{Theorem A}
\newtheorem{definition}{Definition}

\newcommand\ov{\overline}

\mathchardef\mhyphen="2D

\begin{document}

\title{Non-convergence of the spherical harmonic expansion of gravitational potential below the Brillouin sphere; the continuous case}
\author{C. Ogle$^1$, O. Costin$^1$, M. Bevis$^2$}
\date{$^1$ Department of Mathematics, The Ohio State University\\$^2$ Division of Geodetic Science, The Ohio State University\\ \today}
\maketitle

\begin{abstract} For a singleton planet $P$ with gravitational potential $V$, we show that for each $\varepsilon > 0$ there exists a planet $P'$ with gravitational potential $V'$, with $(P',V')$ ``$\varepsilon$-close" to $(P,V)$ (in an appropriate $C^0$-sense) for which the spherical harmonic expansion of $V'$ does not extend more than a distance $\varepsilon$ below the Brillouin sphere of $P'$.  
\end{abstract}

\section*{Introduction} A central problem in geodesy involves computing the gravitational potential $V$ (or its radial derivative gravity) on and above the topography of the Earth based on an initial discrete and finite set of given measurements. The standard method for doing this has involved using that initial data to compute the coefficients of the spherical harmonic expansion of $V$ (which we will denote by $SHE(V)$). The main drawback with this approach is that this expansion is represented by a series which is only guaranteed to converge outside of the Brillouin Sphere \cite{ccob}. A fundamental open question regarding this has been: does $SHE(V)$ converge in the free space within the Brillouin sphere - the smallest sphere which circumscribes the condensed matter that comprises the planet - and, if so, where precisely does convergence occur \cite{tk, hm}?
\vskip.2in
More generally, given some $\varepsilon > 0$, we can consider the weaker property that $SHE(V)$ converges within an $\varepsilon$-neighborhood of the Brillouin sphere $S$, or equivalently that the $SHE(V)$ series descends at least a distance $\varepsilon$ below $S$. If it does this, we say the planet has {\it $\varepsilon$-descent property}. Now the set of planets in $\mathbb R^3$ - denoted by $\cal P$ for this introductory discussion - may be naturally viewed as a function space equipped with a family of weighted $L^p$-norm topologies (defined precisely below). We assume only that the topography of the planet, as well as the density function supported by the planet, are continuous. Our result can be summarized as

\begin{thmA} (cf: Thm.~1, Cor.~1) For all $\varepsilon > 0$ the set of planets in $\cal P$ which do not have the $\varepsilon$-descent property is dense in $\cal P$ in all  weighted $L^p$-topologies.
\end{thmA}

Interestingly, the methods used to prove this theorem also can be used to construct elementary examples of planets with non-radially-symmetric topographies for which $SHE_V$ can be extended all the way down to the topography (Theorem 2; more involved examples have previously appeared in \cite{hm}).
\vskip.2in

In the presence of a slightly stronger condition beyond simple continuity, much more definitive results may be obtained, as we show in \cite{ccob}.
\vskip.5in


\section*{Terminology and notation} We recall that $C^0_c(\mathbb R^3)$ denotes the (topological) vector space of continuous real-valued functions on $\mathbb R^3$ with compact support. An element of $C^0_c(\mathbb R^3)$ is conventionally referred to as a {\it bump function} on $\mathbb R^3$. We write $C^{0,+}_c(\mathbb R^3)\subset C^0_c(\mathbb R^3)$ for the subspace consisting of  those $f:\mathbb R^3\to\mathbb R$ with $Range(f)\subset\mathbb R_+ = \{x\in\mathbb R\ |\ x\ge 0\}$, and $C^{0,+}_{cc}(\mathbb R^3)$ for the subspace of $C^{0,+}_c(\mathbb R^3)$ consisting of those functions with 
\begin{itemize}
\item compact, connected support for which 
\item the boundary of the support consists of a finite disjoint union of closed, compact, connected 2-dimensional topological submanifolds of $\mathbb R^3$, and for which
\item the interior of the support is non-empty and contains the origin.
\end{itemize}
\vskip.2in

For the purposes of modelling problems involving gravity for a singleton planet, the space of functions $C^{0,+}_{cc}(\mathbb R^3)$ provides a suitably general representation in which the planet is represented by a continuous mass-density function on $\mathbb R^3$ with compact, connected support corresponding to the physical mass of the planet. In particular, we do not assume any degree of differentiability, either for the density function or for the boundary of its support.
\vskip.2in
Following convention, we write $B_r({\bf x})$ for the open ball of radius $r$ centered at ${\bf x}\in\mathbb R^3$, $\ov{B}_r({\bf x})$ for its closure, with $S_r({\bf x}) = \partial\ov{B}_r({\bf x})$ its spherical boundary. If $f\in C^{0,+}_{cc}(\mathbb R^3)$, we set $K_f := supp(f)$. By assumption, the boundary of $\partial(K_f) = K_f\backslash int(K_f)$ can be expressed as a disjoint union of closed, compact surfaces
\begin{equation}\label{eqn:boundary}
\partial(K_f) = S_1\sqcup S_2\sqcup\dots\sqcup S_m
\end{equation}
By the generalized Jordan-Brouwer Separation Theorem, each closed 2-dim.~submanifold $S_i$ separates $\mathbb R^3$ into two path-connected components The {\it topography} $T(K_f)$ of $K_f$ is then defined to be the component $S_t$ such that in the complement $\mathbb R^3\backslash S_t$, one component contains the point at $\infty$, while the other component contains $\partial(K_f)\backslash S_t$. In other words, $S_t$ is the extremal path component of $\partial(K_f)$. For most actual planets, $\partial(K_f) = S_1\cong S^2$.
\vskip.2in
Now define
\[
R(f) = \inf\{m\ |\ K_f\subset B_m({\bf 0})\}
\]
As $K_f$ is compact, one has
\[
\inf\{m\ |\ K_f\subset B_m({\bf 0})\} = \min\{m\ |\ K_f\subset \ov{B}_m({\bf 0})\} = \min\{m\ |\ T(K_f)\subset \ov{B}_m({\bf 0})\}
\]
The {\it Brillouin sphere} associated with $f$ is then defined to be $S^2_{R(f)} := S_{R(f)}({\bf 0})$; it is the smallest sphere centered at $\bf 0$ containing $K_f$ (or, equivalently, $T(K_f)$). We occasionally abbreviate this as $S^2_f$ when there is no confusion.
\vskip.2in

We also need to say something about norm topologies on the function space $C^0_c(\mathbb R^3)$. The most basic is the unweighted $L^1$-norm
\[
\|f\|_1 := \int_{\mathbb R^3} |f({\bf x})|d{\bf x} = \int_{K_f} |f({\bf x})|d{\bf x}
\]
and the corresponding $L^p$-norms ($p > 1$)
\[
\|f\|_p := \left(\int_{\mathbb R^3} |f({\bf x})|^pd{\bf x}\right)^{1/p} = \left(\int_{K_f} |f({\bf x})|^pd{\bf x}\right)^{1/p}
\]
which for $p=\infty$ should be interpreted as the sup norm
\[
\|f\|_\infty = \sup_{{\bf x}\in K_f}|f({\bf x})|
\]
More generally, given a continuous function $w:\mathbb R^3\to \mathbb R^+$, we define the $w$-weighted $L^p$-norm ($1\le p\le \infty$) by
\[
\|f\|_{p,w} := \left(\int_{\mathbb R^3} |f({\bf x})|^pw({\bf x})d{\bf x}\right)^{1/p} = \left(\int_{K_f} |f({\bf x})|^pw({x})d{\bf x}\right)^{1/p}
\]
We write $\mu_{p,w}(_-,_-)$ for the corresponding metric: $\mu_{p,w}(f_1,f_2) := \|f_1 - f_2\|_{p,w}$. The {\it $L^{p,w}$-topology} on either $C^0_c(\mathbb R^3)$ or the subspace $C^{0,+}_{cc}(\mathbb R^3)$ will then refer to the metric topology induced by $\mu_{p,w}$. When $w = 1$ is the constant function $1$, we refer to the $L^{p,w}$-topology simply as the {\it $L^p$-topology}.
\vskip.2in
If $f\in C^0(\mathbb R^3)$ is a continuous function on $\mathbb R^3$ with not necessarily compact support, then $\|f\|_{p,w}$ need not be finite. This issue is typically rectified by computing the $\|_-\|_{p,w}$-norm on the restriction of the function to some compact set. For our purposes it will suffice to restrict to closed balls centered at the origin. Thus, for a given real number $N > 0$, let $\chi_N$ denote the characteristic function of $\ov{B}_N({\bf 0})$; this defines a semi-norm $\|_-\|_{N,p,w}$ and corresponding pseudo-metric $\mu_{N,p,w}(_-,_-)$ by
\[
\|f\|_{N,p,w} := \|\chi_N\cdot f\|_{p,w},\quad \mu_{N,p,w}(f,f') := \|f - f'\|_{N,p,w}
\]
Also, for any compact $K\subset\mathbb R^3$ we set
\begin{gather*}
var(f,K) := inf\{(b-a)\ |\ f(K)\subset [a,b]\subset\mathbb R\}\\
mean(f,K) = \frac{\int_K f({\bf x})\ d{\bf x}}{|K|}
\end{gather*}
where $|K| = \int_K 1\;d{\bf x} < \infty$ denotes the volume of $K$.
\vskip.2in

Recall that if $(X,\mu)$ is a metric space and $K\subset X$ a non-empty subspace, then an $\varepsilon$-neighborhood of $K$ is given as
\[
N_{\varepsilon}(K) := \{ x\in X\ |\ \exists y\in K\,\text{with }\, \mu(x,y) < \varepsilon\}
\]
\begin{definition} If $(X,\mu)$ is a metric space, $K,L\subset X$ are two non-empty subspaces of $X$, and $\varepsilon > 0$, we say $K$ and $L$ are $\varepsilon$-close iff $K\subset N_{\varepsilon}(L)$ and $L\subset N_{\varepsilon}(K)$. The distance between $K$ and $L$ is then given as
\[
d(K,L) := \inf\{\varepsilon\ |\ K,L\,\text{are }\varepsilon\text{-close}\}
\]
\end{definition}

An easy argument in point-set topology shows that

\begin{proposition} If $K,L$ and $X$ are as in the above definition with $K,L$ compact, then $d(K,L) = 0$ iff $K=L$.
\end{proposition}

This proposition can be extended to show that this distance function yields a well-defined metric on the space of compact subspaces of $X$. 
\vskip.2in

For a given mass-density function $f\in C^0_{cc}(\mathbb R^3)$, $V_f$ will denote its gravitational potential (again viewed as a function on all of $\mathbb R^3$), and $SHE_N(V_f)$ the truncation at degree $(-N-1)$ in the radial coordinate of the spherical harmonic expansion of $V_f$ centered at $\infty$:
\[
SHE_N(V_f) = \frac{GM}{R}\sum_{n=0}^N\sum_{m= -n}^n \left(\frac{R}{r}\right)^{n+1} C_{n,m}\ov{Y}_{n,m}(\theta,\phi)
\]
where the coefficients $C_{n,m}$ are computed in terms of the mass-density function $f$ \cite{hm}. 
The corresponding series occurs as the limit
\[
SHE(V_f) = \underset{n}{\varinjlim}\, SHE_n(V_f) = \frac{GM}{R}\sum_{n=0}^\infty\sum_{m= -n}^n \left(\frac{R}{r}\right)^{n+1} C_{n,m}\ov{Y}_{n,m}(\theta,\phi)
\]
By \cite{ccob}, $SHE(V_f) = SHE(V_f)(r,\theta,\phi)$ converges uniformly for all $r> R(f)$. For $\varepsilon > 0$ we will say that the series {\it $\varepsilon$-descends} if $SHE(V_f)(r,\theta,\phi)$ converges uniformly for all $r > R(f)-\varepsilon$ in the free space bounded by $S^2_{R(f)}$ and $T(K_f)$. This is equivalent to saying that $SHE_{V_f}$ converges uniformly in the free space of $K_f$ contained in an open $\varepsilon$-neighborhood of $S_{R(f)}$. The subspace of $C_{cc}^{0,+}(\mathbb R^3)$ consisting of those functions for which $SHE(V_f)$ $\varepsilon$-descends will be denoted by $C_{cc}^{0,+}(\mathbb R^3,\varepsilon)$.
\vskip.2in
It will be useful to describe this property in functional terms. For $f\in C_{cc}^{0,+}(\mathbb R^3)$, we set
\[
R_c(f) = \inf\{r > 0\ |\ SHE(V_f)\text{ converges uniformely on }\mathbb R^3\backslash\ov{B}_r({\bf 0})\}
\]
Then $R_c(f)\le R(f)$. Moreover, $f\in C_{cc}^{0,+}(\mathbb R^3,\varepsilon)$ iff $R_c(f)\le R(f) - \varepsilon$.
\vskip.5in


\section*{Statement of the main results} 

\begin{theorem} For any $f\in C^{0,+}_{cc}(\mathbb R^3)$, $1\le p \le \infty$, smooth function $w:\mathbb R^3\to \mathbb R^+$, $N\in\mathbb R^+$, and $\varepsilon > 0$, there is an $f_\varepsilon\in C^{0,+}_{cc}(\mathbb R^3)$ satisfying
\begin{itemize}
\item[(a)] $\mu_{p,w}(f,f_\varepsilon) < \varepsilon$;
\item[(b)] $d(K_{f_\varepsilon},K_f) <\varepsilon$ and $d(\partial K_{f_\varepsilon},\partial K_f) <\varepsilon$;
\item[(c)] $\mu_{N,p,w}(V_{f}(_-), V_{f_\varepsilon}(_-)) < \varepsilon$;
\item[(d)] $d(S^2_f, S^2_{f_\varepsilon}) < \varepsilon$;
\item[(e)] $SHE(V_{f_\varepsilon})$ does not $\varepsilon$-descend.
\end{itemize}
\end{theorem}

This theorem admits the following corollaries.

\begin{corollary} For all $1\le p \le \infty$, weight functions $w$ and $\varepsilon > 0$
\[
\emptyset = int(C^{0,+}_{cc}(\mathbb R^3,\varepsilon))\subset C^{0,+}_{cc}(\mathbb R^3)
\]
 in the $L^{p,w}$-topology. Equivalently, its complement $C^{0,+}_{cc}(\mathbb R^3)\backslash C^{\infty,+}_{cc}(\mathbb R^3,\varepsilon)$ is dense in $C^{\infty,+}_{cc}(\mathbb R^3)$ in the $L^{p,w}$-topology.
\end{corollary}

The association $f\mapsto R_c(f)$ defines a map $R_c: C_{cc}^{0,+}(\mathbb R^3)\to\mathbb R_+$. Let $R_{c,\varepsilon}$ denote the restriction of $R_c$ to $C^{0,+}_{cc}(\mathbb R^3,\varepsilon)$, equipped with the induced topology.

\begin{corollary} For all $\varepsilon > 0$, $1\le p < \infty$, and weight functions $w$, $R_{c,\varepsilon}$ is everywhere discontinuous on its domain in the $L^{p,w}$-topology.
\end{corollary}

These corollaries imply, among other things, that the property of being $\varepsilon$-extendable is completely unstable in any $L^{p,w}$-topology; an arbitrarily small deformation of an $\varepsilon$-extendable $f$ (measured in the metric $\mu_{p,w}$) can produce an $f'$ for which the property fails.
\vskip.2in

On the other hand, the constructions used in proving the above theorem also show there is no shortage of examples of where this does happen for significantly non-spherical topographies.

\begin{theorem} There exists an uncountably infinite dimensional and separable Banach submanifold $M\subset C^{0,+}_{cc}(\mathbb R^3)$ such that for all $f\in M$, $\partial K_f$ is homeomorphic to $S^2$ but not radially symmetric, and for which $SHE(V_f)$ is extendable down to the topography of $K_f$.
\end{theorem}
\vskip.5in


\section*{Proof of Theorem 1} Fix a particular choice of $\mu_{p,w}$-metric on $C_{cc}^{0,+}(\mathbb R^3)$.

\begin{claim} Given $\varepsilon > 0$ there exists a $\delta > 0$ such that $\mu_{p,w}(f,f') < \delta$ implies $\left| V_f(_-) - V_{f'}(_-)\right|_{N,p,w}\} < \varepsilon$
\end{claim}

\begin{proof} Taking the constant function $C := |f - f'|_{p,w}$ on $\ov{B}_N({\bf 0})$ we have
\[
 \left| V_f(_-) - V_{f'}(_-)\right|_{N,p,w} \le \left| V_{C}(_-)\right|_{N,p,w}
\]
by the triangle inequality. By the above, it suffices to show there exists $\delta$ such that $C = |f-f'|_{p,w} < \delta$ implies $\left| V_{\psi}(_-)\right|_{N,p,w} < \varepsilon$. But
\[
\left| V_{C}(_-)\right|_{N,p,w} = C\left| V_{1}(_-)\right|_{N,p,w}
\]
The quantity $D = \left| V_{1}(_-)\right|_{N,p,w} > 0$ (with $1$ denoting the constant function) is a constant independent of $f,f'$. Taking $C = \delta/D$ then completes the proof of the claim.
\end{proof}
\vskip.2in

It remains to show that we may find an $f'$ with $\mu_{p,w}(f,f') < \delta$ for which $SHE(V_{f'})$ does not $\varepsilon$-descend. To do this, we will need to introduce a certain construction derived from point-masses. First some notation. The pair $({\bf x},m)$ will denote the point-mass of mass $m$ (in the appropriate units) located at position ${\bf x}\in\mathbb R^3$. 

\begin{definition} For $r > 0$, an  r-smoothing of $({\bf x},m)$ is defined to be a continuous mass-density function $\lambda$ which is radially symmetric about the point $\bf x$, for which $\|\lambda\|_1 = m$ and $supp(\lambda) = \ov{B}_r({\bf x})$.
\end{definition}

The point of this construction is clear. The conditions are simply a mathematical way of saying that $\lambda$ is the mass-density function with support the closed ball of radius $r$ about $\bf x$ which has total mass $m$, and where the mass of the closed $r$-ball is distributed in a continuous and spherically symmetric manner with respect to the center $\bf x$. We will say $\lambda$ is a {\it smoothed point-mass} (SPM) if it is an $r$-smoothing of some point-mass in the above sense. Finally, we will say that $\lambda$ is a {\it smoothed point-mass array} (SPMA) if $\lambda$ can be written as a finite sum
\[
\lambda = \sum_{i=1}^m\lambda_i
\]
where $\lambda_i$ is an SPM for each $i$.
\vskip.2in

\begin{claim} For any $f\in C_{cc}^{0,+}(\mathbb R^3), 1\le p\le\infty$, weight function $w:\mathbb R^3\to\mathbb R$, and $\delta, \varepsilon > 0$ as in the previous Claim, there is an SPMA $\lambda\in C^{0,+}_{cc}(\mathbb R^3)$ with $\mu_{p,w}(f,\lambda) < \delta$ for which
\[
\max\{d(K_f,K_{\lambda}), d(\partial K_f,\partial K_{\lambda}), d(S^2_f,S^2_{\lambda})\} < \varepsilon
\]
and where $SHE(V_{\lambda})$ does not $\varepsilon$-descend.
\end{claim}

\begin{proof} As the $L^{p,w}$-norm is dominated by the $L^{1,w}$-norm for all $p > 1$, it suffices to consider the case $p=1$. We can further reduce to the case $w = 1$ (the constant function $1$) by observing that
\[
\mu_{1,w}(f,g) = \mu_1(w\cdot f,w\cdot g)
\]
where $\mu_1$ is the standard unweighted $L^1$-metric. With respect to this metric, we claim there exists an SPMA $\lambda = \sum_{i=1}^m \lambda_i$ (where, for each $i$, $\lambda_i$ is an $r_i$-smoothing of a point-mass $({\bf x}_i,m_i)$) with the following properties
\begin{enumerate}
\item[p1)] $\lambda\in C_{cc}^{0,+}(\mathbb R^3)$;
\item[p2)] $\{B_{r_i}({\bf x}_i)\}_{i=1}^m$ is an open covering of $K_f$;
\item[p3)] $\mu_{1}(f,\lambda) < \delta$;
\item[p4)] $d(K_f,K_{\lambda}) < \varepsilon$;
\item[p5)] $d(\partial K_f,\partial K_{\lambda}) < \varepsilon$;
\item[p6)] $d(S^2_f,S^2_{\lambda}) < \varepsilon$;
\item[p7)] the centers ${\bf x}_i$ are in general position with respect to distance from the origin: $\|{\bf x}_i\|\ne \|{\bf x}_j\|$ for all $i\ne j$.
\end{enumerate}
The construction is in stages. A {\it spherical filling} of $K_f$ will refer to a collection of closed balls ${\cal F} := \{\ov{B}_{r_i}({\bf y_i})\}$ with disjoint interiors and where $\ov{B}_{r_i}({\bf y_i})\subset K_f$ for each $i$. Given such a filling $\cal F$, let $K_{\cal F} = \bigcup_i \ov{B}_{r_i}({\bf y_i})\subset K_f$. As $K_f$ is compact with boundary a closed, compact $C^0$ submanifold of $\mathbb R^3$, we can choose a finite spherical filling ${\cal F} = \{\ov{B}_{r_i}({\bf x_i})\}_{i=1}^M$ of $K_f$ for which
\begin{enumerate}
\item[a1)] $\int_{K_f\backslash K_{\cal F}} |f({\bf x})|\ d{\bf x} < \min\{\delta,\varepsilon\}/10$;
\item[a2)] $var(f,\ov{B}_{r_i}({\bf y_i})) < \min\{\delta,\varepsilon\}/(10|K_f|)$;
\end{enumerate}
To this filling we then add a finite collection of balls $\{B_{r_i}({\bf x}_i)\}_{i = M+1}^{M'}$ covering the interstices of the filling, so that
\begin{enumerate}
\item[a3)] $\{B_{r_i}({\bf x}_i)\}_{i=1}^{M'}$ is an open covering of $K_f$;
\item[a4)] setting $K' = \bigcup_{i=1}^{M'}\ov{B}_{r_i}({\bf x}_i)$, $\partial K'$ is given as in (\ref{eqn:boundary}), $d(K_f,K') < \varepsilon$, and $d(\partial K_f,\partial K') < \varepsilon$;
\item[a5)] $d(S^2_f,S^2_{K'}) < \varepsilon$ where $S^2_{K'}$ denotes the Brillouin sphere of the compact region $K'$.
\end{enumerate}
Now for each $1\le i\le M'$ we choose $\lambda_i\in C_{cc}^{0,+}(\mathbb R^3)$ with
\begin{enumerate}
\item[a6)] $supp(\lambda_i) = \ov{B}_{r_i}({\bf x}_i)$;
\item[a7)] for $1\le i\le M$, $\lambda_i$ is radially symmetric about ${\bf x}_i$ and $\mu_1(f_i,\lambda_i) < var(f,\ov{B}_{r_i}({\bf x}_i))$ where $f_i := f|_{\ov{B}_{r_i}({\bf x}_i)}$;
\item[a8)] for $(M+1)\le i\le M'$, $\lambda_i$ is radially symmetric about ${\bf x}_i$ and $var(\lambda_i,\ov{B}_{r_i}({\bf x}_i)) < mean(f,\ov{B}_{r_i}({\bf x}_i))$.
\end{enumerate}
Setting $m_i = \|\lambda_i\|_1 = \int_{\ov{B}_{r_i}({\bf x}_i)}\lambda_i({\bf x})\ d{\bf x}$, we have that $\lambda_i$ is an $r_i$-smoothing of the point-mass $({\bf x}_i,m_i)$ for $1\le i\le M'$. Setting $\lambda = \sum_{i=1}^{M'}\lambda_i$, a1) - a8) imply $\lambda$ satisfies properties p1) - p6) above. The final property p7) is then achieved by a suitably small perturbation of the array of center points $\{{\bf x}_i\}$.
\vskip.2in
Relabeling as needed, we can assume ${\bf x}_1$ is the point in the array maximally distant from the origin. Again, by increasing the  number of balls in our covering if necessary, we can further arrange that $r_1$ - the radius of the $supp(\lambda_1) = \ov{B}_{r_1}({\bf x}_1)$ - is also less than $\varepsilon/2$.
\vskip.2in

By Newton's Theorem, the gravitational potential of $\lambda$ is exactly equal to the potential associated to the finite array of point-masses $S := \{({\bf x}_i, m_i)\}$ everywhere on and outside of the Brillouin sphere $S^2_\lambda$. As the set $S$ of point-masses contains a unique extremal point, the spherical harmonic expansion $SHE(V_S)$ of the gravitational potential associated to the suite of point-masses $S$ is not $\alpha$-extendable for any $\alpha > 0$ below the Brillouin sphere for $S$. To see why this is so, let $S_1 = S\backslash \{({\bf x}_1,m_1)\}$. Let $R$ be the Brillouin radius of $S$, and $R_1$ the Brillouin radius of $S_1$. Let $\delta = R-R_1 > 0$ and set $Sh = \ov{B}_R({\bf 0})\backslash \ov{B}_{R_1}({\bf 0})$; this is a spherical shell centered at $\bf 0$ of thickness $\delta$. $Sh$ is a subset of the free space above the Brillouin sphere of $S_1$, and $SHE(V_{S_1})$ converges absolutely in this region. If there were a point ${\bf y}\in Sh$ where $SHE(V_S)$ converged, then by superposition this would imply
\[
SHE(V_{({\bf x_1},m_1)}) = SHE(V_S) - SHE(V_{S_1})
\]
also converges at this point. But it is well-known that the SHE for a single point-mass does not converge either on or below the Brillouin sphere for the point-mass, leading to a contradiction. It follows, then, that $SHE(V_S)$ cannot coverge anywhere in the region $Sh$.
\vskip.2in

Now $SHE(V_S) = SHE(V_\lambda)$ on and above $S^2_{\lambda}$, and as analytic extensions of harmonic functions are unique, $SHE(V_\lambda)$ is at most extendable down to the Brillouin sphere $S^2_S$ (a distance less than $\varepsilon/2$) but not beyond, as any greater distance would contradict the fact just shown that $SHE(V_S)$ is not extendable on or below its Brillouin sphere. More precisely, $SHE(V_\lambda)$ cannot converge at any point in the free space above the topography $T(K_\lambda)$ that lies in the shell $Sh$ defined above.
\end{proof}

Taking $f_\varepsilon = \lambda$ then completes the proof of Theorem 1.
\newpage


\section*{Proof of Theorem 2} We will construct an SPMA satisfying the properties described in the statement of Theorem 2. Let ${\bf x}_1 = \langle 1,0,0\rangle, {\bf x}_2 = \langle -1,0,0\rangle$. Fix values $m_1,  m_2 > 0$. Next, let $\lambda_i$ be an $r_i$-smoothing of $({\bf x}_i,m_i); i=1,2$ where $r_i = r_i(\gamma) = 1 + \gamma$ for $\gamma > 0$. Let $\lambda = \lambda_1 + \lambda_2$. 
\vskip.2in

We note first that as $\gamma > 0$, $K_\lambda$ satisfies the conditions necessary for $\lambda$ to be an element of $C_{cc}^{0,+}(\mathbb R^3)$. In particular, $\partial K_\lambda$ is homeomorphic to $S^2$, and equal to $T(K_\lambda)$. Set $S = \{({\bf x}_1,m_1), ({\bf x}_2,m_2)\}$. As we have seen previously, $SHE(V_\lambda) = SHE(V_S)$ on and beyond the Brillouin sphere $S_\lambda$, which by construction has radius $(2+\gamma)$. As the Brillouin sphere $B_M$ (see below) of the 2-point-mass array $S$ has radius $1$, we already see that this setup provides a simple example of an SPMA $\lambda$ for which $SHE_{V_\lambda}$ is $(1+\gamma)$-descendable.

\begin{figure}[H]\begin{center}
\includegraphics[scale=0.5]{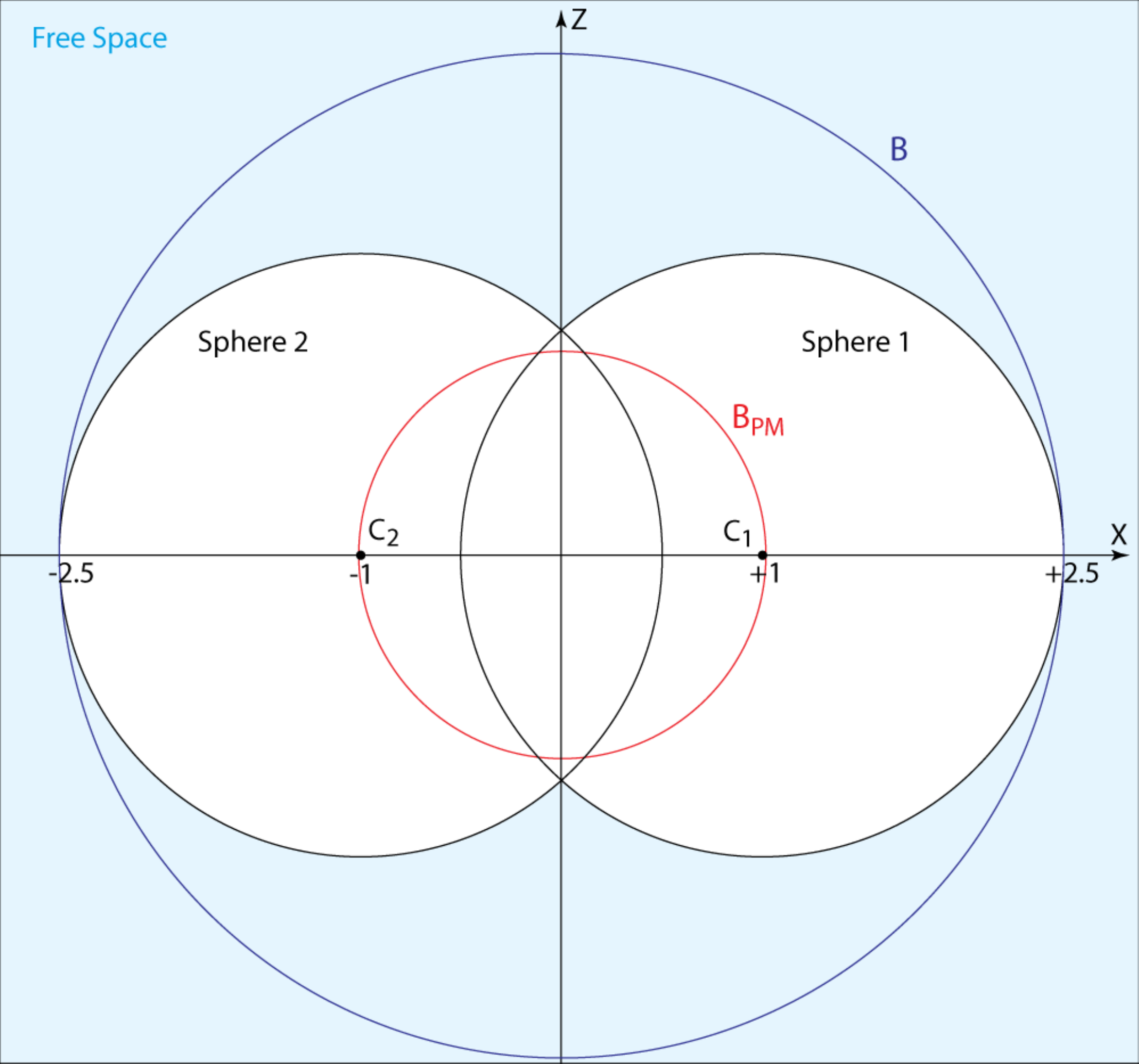}\end{center}
\caption{The Brillouin sphere $B_{PM}$ for the two-point mass array has radius 1, while the Brillouin sphere $B$ for the SPMA has radius 2.5. The topography $T(K_\lambda)$ formed as the union of the two spheres is closest to the origin at their common intersection in the $yz$-plane, which constitutes a circle centered at the origin of radius $\sqrt{5/4} > 1$, and which is therefore exterior to $B_{PM}$.}
\end{figure}
The intersection of $\partial K_\lambda$ with the $yz$-plane is a circle centered at the origin with radius $r(\gamma) = \sqrt{(1+\gamma)^2 - 1}$. This radius grows monotonically with $\gamma$, and $r(\gamma) > 1$ for $\gamma > \sqrt{2} - 1\approx 0.4142$. In particular, taking $\gamma = 1/2$ produces an SPMA $\lambda$ for which the entire topography $T(K_\lambda)$ lies in the exterior of the Brillouin sphere for $S$. It follows from the above discussion that $SHE(V_\lambda)$ is extendable all the way down to $T(K_\lambda)$.
\vskip.2in

We observe that the above argument is independent of the choice of smoothings of $(x_i, {\bf m}_i)$, as well as the choice of positive values for ${\bf m}_i$. For $a > 0$ let $C_1[0,a] = \{f\in C[0,a]\ |\ f(1) = 0\}$, $C^+[0,a] = \{f\in C[0,a]\ |\ f(x) > 0, 0\le x < a\}$, and $C_1^+[0,a] = C_1[0,a]\cap C^+[0,a]$. For $\mu$ the standard Lebesque measure on $\mathbb R^3$ and $f\in C_1^+[0,a]$, the product $\mu_f$ given by $\mu_f({\bf x}) = f(|{\bf x}|)\mu({\bf x})$ defines a radially symmetric measure on $B_a({\bf 0})$ and thus an $a$-smoothing of the point-mass $({\bf 0},m_f)$ where $m_f = \displaystyle\int_{B_a({\bf 0})} f({\bf x})\mu({\bf x})$. Via translation this gives an $a$-smoothing of the point-mass $({\bf b},m_f)$ for any ${\bf b}\in\mathbb R^3$, or simply an $a$-smoothing centered at $\bf b$. Denoting by $Sm(a,{\bf b})$ the set of all $a$-smoothings centered at ${\bf b}$ (of arbitrary positive mass). The above construction yields a canonical isomorphism
\[
Sm(a,{\bf b})\cong C_1^+[0,a]
\]
The space $C_1[0,a]$ is an infinite-dimensional Banach subspace of the Banach space $C[0,a]$, with $C_1^+[0,a]$ open in $C_1[0,a]$. In this way $Sm(a,{\bf b})$ inherits the structure of an infinite-dimensional, separable Banach manifold via the above isomorphism. The product Banach manifold
\[
S = Sm(3/2,{\bf x}_1)\times Sm(3/2,{\bf x}_2)
\]
is then the parameter space of SPMAs with topography represented by the above configuration, in which the exterior SHE will always converge all the way down to the topography.

\end{document}